\documentclass{scrartcl}
\usepackage[utf8]{inputenc}
\usepackage{amsfonts, amssymb, amsthm, amsmath, mathabx, mathtools}
\usepackage[all]{xy}
\usepackage{tikz-cd}
\usepackage{pgfplots}
\usepackage{soul} 
\usepackage[colorinlistoftodos]{todonotes}
\usepackage{appendix}
\usepackage[T1]{fontenc}
\usepackage{aliascnt}
\usepackage{quiver}
\usepackage{graphicx}
\definecolor{greencolor}{rgb}{0,0.45,0}
\definecolor{ucphcolor}{rgb}{0.517,0.016,0.016}
\usepackage[colorlinks,
            citecolor=greencolor,
            linkcolor=ucphcolor,
            urlcolor=greencolor, unicode,linktocpage]{hyperref}
\usepackage[capitalize]{cleveref}

\usepackage{spectralsequences}

\usepackage[backend=biber, style=alphabetic, maxnames=10, maxalphanames=3, minalphanames=3]{biblatex}
\usepackage[margin = 1.25in]{geometry}

\usepackage{mathrsfs}
\usepackage{eucal}
\usepackage{scalerel}

\newtheorem{nul}{}[section]

\theoremstyle{definition}
\newtheorem{rmk}[nul]{Remark}
\newtheorem{rec}[nul]{Recollection}
\newtheorem*{war*}{Warning}
\newtheorem{defn}[nul]{Definition}
\newtheorem{ex}[nul]{Example}
\newtheorem*{ques*}{Question}

\newtheorem*{idea*}{Idea}

\newtheorem{ques}[nul]{Question}
\newtheorem*{ans*}{Answer}

\newtheorem{cons}[nul]{Construction}

\AtBeginEnvironment{rmk}{\pushQED{\qed}}
\AtEndEnvironment{rmk}{\popQED\endrmk}
\AtBeginEnvironment{rec}{\pushQED{\qed}}
\AtEndEnvironment{rec}{\popQED\endrec}
\AtBeginEnvironment{war*}{\pushQED{\qed}}
\AtEndEnvironment{war*}{\popQED\endwar*}
\AtBeginEnvironment{defn}{\pushQED{\qed}}
\AtEndEnvironment{defn}{\popQED\enddefn}
\AtBeginEnvironment{ex}{\pushQED{\qed}}
\AtEndEnvironment{ex}{\popQED\endex}
\AtBeginEnvironment{ques}{\pushQED{\qed}}
\AtEndEnvironment{ques}{\popQED\endques}
\AtBeginEnvironment{cons}{\pushQED{\qed}}
\AtEndEnvironment{cons}{\popQED\endcons}

\theoremstyle{plain}
\newtheorem{corr}[nul]{Corollary}
\newtheorem{lem}[nul]{Lemma}
\newtheorem{thm}[nul]{Theorem}
\newtheorem*{thm*}{Theorem}

\graphicspath{{pics/}}

\newcommand{\syn}[1]{\rm{Syn}_{#1}}

\newcommand{\bb}[1]{\mathbb{#1}}

\newcommand{\mc}[1]{\mathcal{#1}}

\newcommand{\modp}{\bb{F}_{p}}

\newcommand{\rm}[1]{\mathrm{#1}}

\newcommand{\map}{\mathrm{map}}
\newcommand{\Map}{\mathrm{Map}}

\def\F{\mathbb{F}}
\def\Ext{\mathrm{Ext}}
\def\CC{\mathcal{C}}
\DeclareMathOperator*{\colim}{colim}

\def\Ss{\mathbb{S}}
\def\Sp{\mathrm{Sp}}
\def\A{\mc{A}}
\def\cofib{\mathrm{cofib}}

\title{A note on the Segal conjecture for large objects}
\date{\today}
\author{Robert Burklund \and Vignesh Subramanian}

\begin{document}

\maketitle

\begin{abstract}
    The Segal conjecture for $C_p$ (as proved by Lin and Gunawardena) asserts that the canonical map from the $p$-complete sphere spectrum to the Tate construction for the trivial action of $C_p$ on the $p$-complete sphere spectrum is an isomorphism. 
    In this article we extend the collection of spectra for which the canonical map $X \to X^{tC_p}$ is known to be an isomorphism to include any $p$-complete, bounded below spectrum whose mod $p$ homology, viewed a module over the Steenrod algebra, is complete with respect to the maximal ideal $I \subseteq \mathcal{A}$. 
\end{abstract}


\section{Introduction}


Inspired by the Atiyah--Segal completion theorem 
relating the completion of the representation ring of a group at the augmentation ideal to Borel equivariant $K$-theory, Segal conjectured that the natural comparison map
$$A(G)^{\wedge}_{I} \to \pi_{0}\left( \bb{S}^{BG_{+}} \right) $$
from the completion of the Burnside ring of a finite group $G$ at its augmentation ideal to 
the zero'th homotopy group of the Borel equivariant sphere spectrum is an isomorphism.
Segal's conjecture spurred a flurry of work on equivariant stable homotopy theory in the late 70s and early 80s \cite{lin1980conjectures, lin1980calculation, gunawardena1980segal, mayMcClure, ravenel1984segal, AGM} ending with Carlsson's proof of the full conjecture \cite{carlsson1984equivariant}.





Specializing to the case of $G = C_p$ isotropy separation relates the original form of Segal's conjecture to the claim that the canonical map $\bb{S}_p \rightarrow \bb{S}_p^{tC_p}$ from the $p$-complete sphere spectrum $\bb{S}_p$ to the Tate construction with respect to trivial $C_p$ on $\bb{S}_p$ is an isomorphism.
It is the latter statement which Lin proved in his breakthrough work \cite{lin1980conjectures} on the case $G=C_2$.

\begin{thm} [\cite{lin1980conjectures, gunawardena1980segal}] \label{linmainthm}
    The canonical map $\bb{S}_{p} \to \bb{S}_p^{tC_p}$ is an isomorphism.
\end{thm}

In this article we investigate the collection of spectra for which the conclusion of the theorem above holds.

\begin{ques*} 
    For which spectra $X$ is the canonical map 
    $ X \to X^{tC_p} $
    an isomorphism?
\end{ques*}

As the Tate construction is an exact functor, \Cref{linmainthm} implies the canonical map is an isomorphism for any $X$ in the thick subcategory generated by $\bb{S}_p$ (i.e. the $p$-completions of finite spectra). In his thesis \cite{yuan2021integral}, Yuan proves that the canonical map is an isomorphism for a $p$-completion of an (infinite) sum of spheres $\left( \bigoplus_{i \in I} \bb{S} \right)_{p}$. Our main theorem extends the collection of $X$ for which this question has a positive answer to include many more examples of bounded below infinite spectra. 



 
\begin{thm} [Theorem~\ref{segalconjIcmpl}]
    Let $X$ be a $p$-complete, bounded below spectrum.
    If the homology of $X$ is $I$-complete (in the sense defined below), then the canonical map
    \[ X \rightarrow X^{tC_p} \]
    is an isomorphism.
\end{thm}

\begin{defn} \label{defnInilpotent}
    Given a spectrum $X$, the $\bb{F}_p$-homology naturally carries an an action of the mod $p$ Steenrod algebra $\mc{A}$. 
    Let $I$ be the augmentation ideal of the Steenrod algebra. 
    We say the homology of $X$ is $I$-complete if 
    $H_{*}(X) \cong \varprojlim_{n} H_{*}(X)/I^{n}$.
\end{defn}

In order to explain our proof of \Cref{segalconjIcmpl} it will be useful to review the method used by Lin to prove \Cref{linmainthm} in \cite{lin1980conjectures}.
The key step is showing that it in fact suffices to prove that the canonical map induces an isomorphism on the level of the $E_2$-page of an appropriate Adams spectral sequence. The subtle point here is that when taking the Adams spectral sequence of $\bb{S}^{BC_p}$ it is important to first filter $BC_p$ by finite skeleta $P^k$, then take the Adams spectral sequence of $\bb{S}^{P^k}$ and then and only then pass to the inverse limit over $k$  at the level of Adams spectral sequences. The catch is that one must make good sense out of the phrase "limit of spectral sequences". Lin implemented this in \cite{lin1980conjectures} in an ad-hoc way, specialized to the problem at hand.
The isomorphism of $E_2$-pages then comes from the ext calculations of \cite{lin1980calculation}.

In this article, we revisit Lin's constructions within the context of synthetic spectra
as a natural home for "limits of synthetic spectra". The upshot of this is that it is now straightforward to prove a reduction to an $E_2$-page calculation for any bounded below $X$. 
After some further categorical manipulations we find that $I$-completeness is enough to reduce the necessary ext calculation to the cases already handled in \cite{lin1980calculation} and \cite{gunawardena1980segal}.

\begin{rmk}
    There are many examples where the canonical map is not an isomorphism.
    For example, 
    $\F_p^{tC_p}$ is unbounded (while $\F_p$ itself is bounded) and
    $(\mathrm{KU}_p)^{tC_p}$ is a $\bb{Q}$-module (while $\mathrm{KU}^{\wedge}_{p}$ is not rational). 
\end{rmk}

\begin{rmk}
    In the unstable setting Miller proved the Sullivan conjecture, that for a finite space $X$ with a trivial $C_p$ action, the map $X \rightarrow X^{hC_p}$ is an isomorphism after $p$-completion \cite{miller1984sullivan}.
    Extending this Lannes and Schwartz proved the same result with the weaker hypothesis that the cohomology of $X$ is locally finite \cite{lannes1986conjectures}. 
    It was this extension of the Sullivan conjecture to non-finite spaces that motivated us to look for a corresponding extension of the Segal conjecture to non-finite spectra.
\end{rmk}

In \Cref{sec:syn} we briefly review the theory of synthetic spectra which we will be using.
In \Cref{sec:main} we prove our main theorem.
Finally, in \Cref{examples} we give several examples of non-finite $I$-complete spectra.
A example of note is $\Sigma_+^\infty X$ where $X$ is a space whose cohomology is finite type and locally finite (see Definition \ref{defnlocallyfinite}). 

\subsection*{Acknowledgments}

The authors would like to thank Jesper Grodal, Haynes Miller and Zhouli Xu for helpful conversations related to the contents of this paper.
The authors would further like to thank Tobias Bartel, Natalia Castellena, Jesper Grodal, Branko Juran and Jesper Michael M{\o}ller for useful comments on a draft.

During the course of this work the authors were supported by the DNRF through the Copenhagen center for Geometry and Topology (DNRF151).
In addition, the first author was supported by NSF grant DMS-2202992.

\section{Synthetic Spectra and the Adams spectral sequence} \label{sec:syn}

In this section we provide a brief refresher on the Adams spectral sequence and synthetic spectra.
Our aim is to set the stage for the final section; our exposition will be terse. 


\begin{defn} 
    Let $E$ be homotopy associative ring spectrum.
    
    A finite spectrum $X$ is said to be \emph{finite $E$-projective} if $E_{*}X$ is a finitely generated, projective module over the graded ring $E_{*}$.
    
    We say that $E$ is \emph{of Adams-type}, 
    if $E$ can be written as a filtered colimit $E \cong \varinjlim E_{\alpha}$, where each $E_{\alpha}$ is finite $E$-projective and the $E$-cohomology of each $E_{\alpha}$ is dual to its homology, i.e. 
    $$ E^{*}E_{\alpha} \cong \rm{Hom}_{E_{*}}(E_{*}E_{\alpha}, E_{*}).\hfill\qedhere $$
\end{defn}


\begin{ex}
    Examples of Adams-type ring spectra include:
    $\bb{Q}$, 
    $\bb{F}_p$,
    the $p$-adic complex $K$-theory spectrum $KU_p$,
    the complex cobordism spectrum $MU$, 
    any Landweber exact cohomology theory and
    the Morava $K$-theories $K(n)$.
\end{ex}

Given an Adams-type ring spectrum $E$, Adams constructed, for each pair of spectra $X, Y \in \mathrm{Sp}$ with $X$ finite $E$-projective, a spectral sequence of signature
$$ \rm{Ext}^{s,t}_{E_{*}E}(E_{*}X,E_{*}Y) \cong E^{2}_{s,t} \implies \pi_{t-s}(\rm{map}(X,Y^{\wedge}_{E}))$$
where the $\rm{Ext}$ groups are calculated in the category of comodules over $E_{*}E$ and $Y^{\wedge}_{E}$ denotes the $E$-nilpotent completion of the spectrum $Y$ \cite{adams1974stable}.  

\begin{ex}
    In the case $E=\bb{F}_p$ 
    we may identify ${\modp}_{*}\modp$ with the mod-$p$ dual Steenrod algebra $\mathcal{A}^{\vee}_{p}$.    
\end{ex}

Building on Adams' work Pstr\k{a}gowski attached to each Adams-type ring spectrum $E$ a category $\syn{E}$ of $E$-synthetic spectra 
categorifying the $E$-based Adams spectral sequence. Objects of $\syn{E}$ are morally ``generalized $E$-Adams spectral sequences''.
For us this is a boon: In the past taking the necessary limits of Adams spectral sequences (as described in the introduction) was a tricky business, today it is as simple as noting that $\syn{E}$ is presentable and therefore complete.
For a general introduction to synthetic spectra we encourage readers to visit the wonderfully well written \cite{pstrągowski2022synthetic} where synthetic spectra where first introduced.
Below we recall the key properties which we will need.


\begin{rec}[{Pstr\k{a}gwoski \cite[Section~4]{pstrągowski2022synthetic}}]\hfill\\ \label{recallsyn}

    The category of $E$-based synthetic spectra $\syn{E}$ is 
    a stable, presentably symmetric monoidal $\infty$-category 
    that sits in a diagram of lax symmetric monoidal functors
    \[\begin{tikzcd}[sep=huge]
	   & {\rm{Sp}} \\
	   & {\syn{E}} \\
	   {\rm{Sp}} && {\rm{Mod}_{C\tau}} \ar[r, hook] & \mathrm{Stable}_{E_{*}E}
	   \arrow["\nu", from=1-2, to=2-2]
	   \arrow["{(-)[\tau^{-1}]}", from=2-2, to=3-1]
	   \arrow["\mathrm{Id}"', bend right, from=1-2, to=3-1]
	   \arrow["{C\tau \otimes-}"', from=2-2, to=3-3]
	   \arrow["{E_{*}(-)}", bend left, from=1-2, to=3-4]
    \end{tikzcd}\]
    \begin{enumerate}
        \item The synthetic analogue functor $\nu$ is fully faithful and preserves filtered colimits.
        \item If a spectrum $X$ can be written as a filtered colimit of finite $E$-projective spectra, then the map 
        $$\nu(X) \otimes \nu(Y) \rightarrow \nu(X \otimes Y)$$
        coming from the lax symmetric monoidal structure on $\nu$ is an isomorphism.
        \item In general $\nu$ does not preserve all colimits. Indeed,
        given a cofiber sequence $X \rightarrow Y \rightarrow Z$ in $\rm{Sp}$, 
        the associated sequence $\nu X \rightarrow \nu Y \rightarrow \nu Z$ is a cofiber sequence if and only if 
        \[ 0 \rightarrow E_{*}(X) \rightarrow E_{*}(Y) \rightarrow E_{*}(Z) \rightarrow 0 \]
        is a short exact sequence of $E_*E$ comodules.
        \item We define the bigraded spheres in $\syn{E}$ by 
        $\bb{S}^{a,b} : = \Sigma^{-b}\nu(\bb{S}^{a+b}) $
        and define the bigraded synthetic homotopy groups of an $X \in \syn{E}$ by
        \[ \pi_{a,b}(X) \coloneqq \pi_{0}\map_{\syn{E}}(\bb{S}^{a,b}, X). \]
        \item There is a map $\tau : \bb{S}^{0,-1} \rightarrow \bb{S}^{0,0}$ 
        constructed as the assembly map $\Sigma \nu(\mathbb{S}^{-1}) \rightarrow \nu(\Sigma \mathbb{S}^{-1} ) $. Tensoring this map with $X \in \syn{E}$ give an action of $\tau$ on $X$.
        \item The inclusion of $\tau$-invertible synthetic spectra into all synthetic spectra 
        $\syn{E}[\tau^{-1}] \hookrightarrow \syn{E}$
        admits a left adjoint 
        $(-)[\tau^{-1}]: \syn{E} \rightarrow \syn{E}[\tau^{-1}]$. 
        Given a synthetic spectrum $X$ then the $\tau$-inversion is given by
        $$X[\tau^{-1}] = \varinjlim ( X \xrightarrow{\tau} \Sigma^{0,1} X \xrightarrow{\tau} \Sigma^{0,2}X \xrightarrow{\tau}...)$$
        the colimit over $\tau$. 
        In particular, inverting $\tau$ is a smashing localization given by $- \otimes \bb{S}^{0,0}[\tau^{-1}]$.
        \item The composition of functors $\rm{Sp} \xrightarrow{\nu} \rm{Syn}_E \xrightarrow{(-)[\tau^{-1}]} \syn{E}[\tau^{-1}]$ is an isomorphism and identifies $\syn{E}[\tau^{-1}]$ with the category of spectra $\rm{Sp}$.
        \item The cofiber of $\tau$, $C\tau$ has a unique $\bb{E}_\infty$-algebra structure.
        \item For $E$ a homotopy commutative ring spectrum, then there is a fully faithful inclusion 
        \[ \rm{Mod}_{C\tau} \hookrightarrow \rm{Stable}_{E_*E} \] 
        from the category of $C\tau$-modules to the stable $\infty$-category of $E_*E$-comodules (in the sense of Hovey)\footnote{See \cite[Section~3.2]{pstrągowski2022synthetic}.} 
        such that the diagram above commutes.
        As a consequence for $X \in \rm{Sp}$ we have a natural isomorphism
        $$\pi_{t-s,s}(C\tau \otimes \nu X) \cong \rm{Ext}^{s,t}_{{E}_{*} E}(E_*, E_{*}(X)).$$
        The $\tau$-bockstein sseq then coincides with the $E$-based Adams sseq (see \cite[Section~9]{burklund2022boundaries}).
        \item We say a synthetic spectrum $X$ is $\tau$-complete if the map 
        $X \rightarrow \varprojlim_{n} X \otimes C \tau^n$ is an isomorphism.\footnote{This is same as being complete with respect to the dualizable algebra $C\tau$ in the sense of \cite[Section~2]{MNN17}.} Let $Y$ be a spectrum. $Y$ is $E$-nilpotent complete if and only if $\nu Y$ is $\tau$-complete (see \cite[Prop~A.13]{burklund2022boundaries}).
        Note that on $\tau$-complete synthetic spectra the functor $C\tau \otimes -$ is conservative.\hfill\qedhere
    \end{enumerate}
\end{rec}

We summarize the above by saying that we think of $\syn{E}$ as a one-parameter deformation of the category of spectra with deformation parameter $\tau$. At $\tau = 0$ the deformation degenerates and we obtain an algebraic special fiber (given by a category of comodules).




\begin{rmk} 
    The case $E=\bb{F}_p$ has several special features.
    \begin{enumerate}
        \item Every finite spectrum is finite $\bb{F}_{p}$-projective. 
        \item $\nu_{\modp}$ is symmetric monoidal.
        \item The bigraded spheres $\bb{S}^{a,b}$ are a collection of compact generators of $\syn{\bb{F}_p}$.
        \item A map of synthetic spectra $f: X \rightarrow Y$ is an isomorphism if and only if it induces an isomorphism of bigraded homotopy groups.\hfill\qedhere
    \end{enumerate}
\end{rmk}

In this article we will work exclusively in the case $E = \modp$.

\begin{rmk}\label{rmktaucompl}
    Let $X$ be a $p$-complete, bounded below spectrum, then $X$ is $\bb{F}_p$-nilpotent complete \cite{bousfield1979localization}.
    In particular, $\nu X$ is $\tau$-complete if $X$ is $p$-complete and bounded below.
\end{rmk}

It is this $\tau$-completeness that will allow us to check that various maps are isomorphisms mod $\tau$ (i.e. in some category of comodules) and draw conclusions in the category of synthetic spectra.



In preparation for the proof of our main theorem we will need a somewhat technical lemma about the interaction between inverting $\tau$ and limits.

\begin{lem} \label{lem:lim-tau-invert-commute}
    Let $F : J \to \Sp$ be a diagram of spectra. The coassembly map 
    \[  \left( \lim_J \nu(F(j)) \right)[\tau^{-1}] \to \lim_J \left( \nu(F(j))[\tau^{-1}] \right) \cong \lim_J F(j) \]
    is an isomorphism.
\end{lem}

\begin{proof}
    Using compactness of the bigraded spheres and the fact that $\nu$ is fully faithful we compute that
    \begin{align*} 
        \Map_{\Sp} & \left(\Ss^n, \left( \lim_J \nu F(j) \right)[\tau^{-1}] \right) 
        \cong \colim_{0 \leq t \to \infty} \Map_{\syn{E}}(\Sigma^{0,-t}\nu\Ss^{n}, \lim_J \nu F(j)) \\
        &\cong \colim_{0 \leq t \to \infty} \lim_J \Map_{\syn{E}}(\Sigma^{0,-t}\nu\Ss^{n}, \nu F(j))
        \cong \colim_{0 \leq t \to \infty} \lim_J \Omega^t \Map_{\syn{E}}(\nu\Ss^{n-t}, \nu F(j)) \\
        &\cong \colim_{0 \leq t \to \infty} \lim_J \Omega^t \Map_{\Sp}(\Ss^{n-t}, \F(j))
        \cong \colim_{0 \leq t \to \infty} \lim_J \Map_{\Sp}(\Ss^{n}, \F(j)) \\
        &\cong \lim_J \Map_{\Sp}(\Ss^{n}, \F(j))
        \cong \Map_{\Sp}(\Ss^{n}, \lim_J \F(j)).
    \end{align*}
\end{proof}

\section{The Segal Conjecture for $I$-complete spectra} \label{sec:main}

Following the same broad strokes as in \cite{lin1980conjectures} we will prove our main theorem by rewriting the Tate construction as a limit, evaluating this limit at the level of Adams spectral sequences (implemented for us through $\F_p$-synthetic spectra) and then proving that induced map on $E_2$-pages is already an isomorphism.

\begin{cons}[Stunted lens spaces] \label{stuntedprojspectrum}
    Let $\lambda$ be the rank $1$ complex representation of $C_p$ where the generator $\sigma \in C_p$ acts by a primitive $p$'th root of unity. Using this we form the stable representation spheres $\bb{S}^{k\lambda} \in \mathrm{Sp}^{BC_p}$ for each $k \in \bb{Z}$. The inclusions $k\lambda \to (k+1)\lambda$ provide $C_p$-equivariant maps of representation spheres $a_\lambda : \bb{S}^{k\lambda} \to \bb{S}^{(k+1)\lambda}$. We define $ P^{\infty}_{k} \coloneqq (\bb{S}^{k\lambda})_{hC_p} $ and write $a_\lambda$ for the corresponding maps $P^\infty_k \to P^\infty_{k+1}$.
\end{cons}

When $p=2$ the stunted lens spaces $P_k^\infty$ coincide with James' stunted projective spaces $\bb{R}P^\infty_{2k}$ \cite{james1959spaces} (cf. \cite{atiyah1961thom}). The Thom isomorphism lets us compute that the homology of $P_k^\infty$ is $\F_p$ in degrees $\geq 2k$ and zero otherwise.


\begin{lem} \label{tatelimit}
    Using the stunted lens spaces we have an expression for the Tate construction of the trivial action on $X$ as $ \varprojlim_{k} (X \otimes \Sigma P_{k}) $
    where the limit is along the maps induced by $a_\lambda$.
\end{lem}

\begin{proof}
    We will prove the lemma by examining the cofiber sequence
    \[ \varprojlim_k (\bb{S}^{k\lambda} \otimes X)_{hC_p} \to \varprojlim_k (\bb{S}^{k\lambda} \otimes X)^{hC_p} \to \varprojlim_k (\bb{S}^{k\lambda} \otimes X)^{tC_p}. \]
    First we note that as $X$ has trivial action 
    $\varprojlim_k (\bb{S}^{k\lambda} \otimes X)_{hC_p} \cong \varprojlim_k (P_k^\infty \otimes X)$.
    Second we rewrite $\varprojlim_k (\bb{S}^{k\lambda} \otimes X)^{hC_p} $ as
    $  ( \varprojlim_k (\bb{S}^{k\lambda} \otimes X) )^{hC_p} $
    and use the fact that the underlying (non-equivariant) map of $a_\lambda$ is nullhomotopic
    to conclude that $ \varprojlim_k (\bb{S}^{k\lambda} \otimes X) = 0$.
    The lemma now follows by noting that the cofiber of $a_\lambda$ is built out of two free $C_p$-cells and using this to conclude that $(a_\lambda \otimes X)^{tC_p}$ is an isomorphism and therefore
    $ \varprojlim_k (\bb{S}^{k\lambda} \otimes X)^{tC_p} \cong X^{tC_p}$.    
\end{proof}



\begin{thm} \label{segalconjIcmpl}
    Let $X$ be a $p$-complete, bounded below spectrum.
    If the homology of $X$ is $I$-complete, then $X$ satisfies the Segal conjecture for $C_p$. 
    That is, the canonical map 
    \[ X \rightarrow X^{tC_p} \] 
    is an isomorphism, where the Tate construction is with respect to trivial action.
\end{thm}
 
\begin{proof}
    The formula for the Tate construction from \Cref{tatelimit}
    together with \Cref{lem:lim-tau-invert-commute} gives us a commutative diagram
    \[ \begin{tikzcd}
        (\nu X)[\tau^{-1}] \ar[r] \ar[d, "\cong"] & 
        \left( \varprojlim_{k} \nu(X \otimes \Sigma P^{\infty}_{k}) \right)[\tau^{-1}] \ar[d, "\cong"] \\
        X \ar[r] &
        X^{tC_p}
    \end{tikzcd} \]
    where the vertical maps are isomorphism.
    We will prove the theorem by showing that the map
    \begin{align} \nu X \to \varprojlim_{k} \nu(X \otimes \Sigma P^{\infty}_{k}) \label{eqn:want-iso} \end{align}
    is already an isomorphism in $\syn{\F_p}$ before inverting $\tau$.

    Using that $X$ is bounded below and $p$-complete we know from \Cref{rmktaucompl} 
    that $X$ is also $\F_p$-nilpotent complete and $\nu X$ is $\tau$-complete.
    Similarly, $X \otimes \Sigma P^\infty_k$ is bounded below and $p$-complete for every $k \in \bb{Z}$, therefore
    $\nu(X \otimes \Sigma P^{\infty}_{k})$ is $\tau$-complete as well.
    Since the subcategory of $\tau$-complete synthetic spectra spectra is closed under limits, 
    $\varprojlim_{k} \nu(X \otimes \Sigma \bb{R}P^{\infty}_{k})$ is $\tau$-complete. 
    The upshot of this is that it will suffice to prove that the map from \Cref{eqn:want-iso}
    is an isomorphism after tensoring with $C\tau$.
    
    Using the fact that $C\tau$ is dualizable\footnote{Recall that $C\tau \cong \rm{Cofib}(\tau : \bb{S}^{0,-1} \rightarrow \bb{S}^{0,0})$, where both $\bb{S}^{0,-1}$ and $\bb{S}^{0,0}$ are dualizable.} to commute $C\tau \otimes -$ past the limit 
    together with the isomorphism $ C\tau \otimes \nu X \cong H_{*}(X)$ and the Kunneth isomorphism 
    we reduce further to showing that the map
    \[ H_{*}(X) \rightarrow \varprojlim_k  \left( H_{*}(X) \otimes H_{*}(\Sigma P^{\infty}_{k}) \right) \]
    is an isomorphism in $\mathrm{Stable}(\mc{A})$.
    
    Using the hypothesis that the homology of $X$ is $I$-complete 
    together with \Cref{lem:colimlim} and \Cref{cor:lim-underlying} below we have isomorphisms
    \[ H_{*}(X) \cong \varprojlim_n H_{*}(X)/I^n, \] 
    \[ H_{*}(X) \otimes H_{*}(\Sigma P^{\infty}_{k}) \cong \varprojlim_n \left( H_{*}(X)/I^n \otimes H_{*}(\Sigma P^{\infty}_{k}) \right) \]
    (there is no $\varprojlim^1$ term since the transition maps in these systems are all surjective).
    This allows us to pass to the associated graded of the $I$-adic filtration on $H_*(X)$ and prove the desired isomorphism there.
    Let $K_n \coloneqq \rm{fib}(H_{*}(X)/I^{n+1} \rightarrow H_{*}(X)/I^{n})$. 
    We have at this point reduced to showing that 
    \[ K_n \rightarrow \varprojlim_k K_n \otimes H_{*}(\Sigma P^{\infty}_{k}) \] 
    is an isomorphism in $\mathrm{Stable}(\mc{A})$ for each $n \geq 0$.
    
    At this point we observe that $\mc{A}_p$ acts trivially on $K_n$
    and therefore that $K_n \cong \bigoplus_{i} \bb{F}_p[i]^{\oplus J_i}$ for a collection of sets $J_i$ with $J_i = \emptyset$ 
    for $i \ll 0$. Using fact that this sum is bounded below we have 
    $\bigoplus_{i} \bb{F}_p[i]^{\oplus J_i} \cong \prod_{i} \bb{F}_p[i]^{\oplus J_i} $.
    We can furthermore write $\bb{F}_p[i]^{\oplus J_i}$ as a retract of $\bb{F}_p[i]^{\times I_i}$. At this point it will suffice to show that
    \[ \prod_{i} \bb{F}_p[i]^{\times J_i} \to \varprojlim_k \left( \left( \prod_{i} \bb{F}_p[i]^{\times J_i} \right) \otimes H_{*}(\Sigma P^{\infty}_{k}) \right) \]
    is an isomorphism in $\mathrm{Stable}(\mc{A})$.
    Applying \Cref{lem:colimlim} again to commute the tensor product with $H_{*}(\Sigma P^{\infty}_{k})$ past this product we reduce to showing that
    $$\bb{F}_p[i] \rightarrow \varprojlim_k \left( \bb{F}_p[i] \otimes H_{*}(\Sigma P^{\infty}_{k}) \right) $$
    is an isomorphism in $\mathrm{Stable}(\mc{A})$.
    That this map is an isomorphism is the Ext-isomorphism of \cite{lin1980conjectures, lin1980calculation, gunawardena1980segal}, which was the key input in proving the Segal conjecture for $X=\Ss$ and $G=C_p$. 
\end{proof}

\begin{lem} \label{lem:colimlim}
    Let $X \in \mathrm{Stable}(\mc{A})^\heartsuit$ be a finite type $\mc{A}$-comodule.
    Let $F : J \to \mathrm{Stable}(\mc{A})^\heartsuit$ be a diagram of $\mc{A}$-comodules where 
    the $F(j)$ are uniformly bounded below (with respect to the internal grading).
    The natural comparison map 
    \[ X \otimes \left( \lim_{J} F(j) \right) \to  \lim_J \left( X \otimes F(j) \right) \]
    in $\mathrm{Stable}(\mc{A})$ is an isomorphism.
\end{lem}

\begin{proof}
    We begin by observing that because the collection of functors $\Ext^{s,t}_{\mc{A}}(\F_p, -)$ are jointly conservative on $\mathrm{Stable}(\mc{A})$ 
    it will suffice to show this map induces isomorphisms on these ext groups.
    Now we make the key observation.
    Let $\CC_m \subseteq \mathrm{Stable}(\mc{A})$ denote the full subcategory consisting of those $Z$ such that $\Ext_{\mc{A}}^{s,t}(\F_p, Z) = 0$ for all $s,t$ with $t < m$.
    We have $\F_p[t] \in \CC_t$.
    As each of the shifts $\F_p[t] \in \mathrm{Stable}(\mc{A})$ is compact and $\Ext^{s,t}(\F_p, Z) \cong \pi_{-s}\Map^{\Sp}_{\mathrm{Stable}(\mc{A})}(\F_p[t], Z)$
    we can read off the following facts:
    \begin{enumerate}
        \item An $\mc{A}$-comodule concentrated in internal degrees $\geq m$ is in $\CC_m$.
        \item $\CC_m$ is closed under limits and colimits.
        \item If $M \in \CC_m$ and $N \in \CC_n$, then $M \otimes N \in \CC_{m+n}$.
    \end{enumerate}
    
    As $X$ is of finite type, we may write it as a colimit of skeleta $\colim_{k} X^{(k)}$ where 
    each $X^{(k)}$ is a finite $\A$-comodule concentrated in the range $[0,k]$ and the map $X^{(k-1)} \to X^{(k)}$ is an isomorphism in internal degree $< k$.
    Let $Y^{(k)} \coloneqq \cofib(X^{(k-1)} \to X)$, then $Y^{(k)}$ is an $\mc{A}$-comodule concentrated in degrees $\geq k$.
    
    At this point we may conclude that there is an $N$ such that
    $Y^{(k)} \otimes \lim_{J} F(j)$ and $\lim_K \left( Y^{(k)} \otimes F(j) \right)$ are both in $\CC_{k+N}$. 
    From this we learn that the vertical maps in the diagram below are isomorphisms for $t < k+N$.
    \[ \begin{tikzcd}
        \Ext_{\mc{A}}^{s,t}(\F_p, X^{(k)} \otimes \left( \lim_{J} F(j) \right) ) \ar[r] \ar[d] &
        \Ext_{\mc{A}}^{s,t}(\F_p, \lim_J \left( X^{(k)} \otimes F(j) \right)) \ar[d] \\
        \Ext_{\mc{A}}^{s,t}(\F_p, X \otimes \left( \lim_{J} F(j) \right)) \ar[r] &
        \Ext_{\mc{A}}^{s,t}(\F_p, \lim_J \left( X \otimes F(j) \right))
    \end{tikzcd} \]
    On the other hand as $X^{(k)}$ is a finite $\mc{A}$-comodule it is dualizable and therefore the top horizontal map is an isomorphism as well.
    Allowing $k$ to vary we learn that the bottom horizontal map is an isomorphism for all $s,t$ as desired.    
\end{proof}

\begin{corr} \label{cor:lim-underlying}
    The forgetful functor $\mathrm{Stable}(\mc{A}) \to \mc{D}(\F_p)^{\mathrm{gr}}$ commutes with 
    limits of diagrams that are uniformly bounded below (with respect to the internal grading) and valued in the heart.
\end{corr}

\begin{proof}
    $\mc{D}(\F_p)^{\mathrm{gr}}$ is isomorphic to modules over $H_*(\F_p)$ in $\mathrm{Stable}(\mc{A})$ and the forgetful functor may therefore be identified with the basechange $H_*(\F_p) \otimes -$.
    The conclusion now follows from the fact that $H_*(\F_p)$ is locally finite and \Cref{lem:colimlim}.
\end{proof}

\section{Examples of $I$-complete spectra} \label{examples}

\begin{defn} \label{defnlocallyfinite}
Given a spectrum $X$, we say that it has \emph{locally finite $\bb{F}_p$-cohomology} if for any element $x \in H^{*}(X)$, the sub- $\mc{A}$-module generated by $x$ is finite.
\end{defn}

\begin{ex}
    Let $X$ be a bounded below spectrum whose $\bb{F}_p$-cohomology is locally finite and of finite type. Then the $\bb{F}_p$-homology of $X$ is $I$-complete.
\end{ex}

\begin{proof}[Details.]    
    We must verify that $H_{*}(X) \cong \varprojlim_n H_{*}(X)/I^n$.
    As $H^*(X)$ is finite type we have $H_*(X) \cong H^*(X)^\vee$.
    The hypothesis that $H^*(X)$ is locally finite means we can write $H^*(X)$ as union of the finite sub- $\mc{A}$-module generated by $H^*(X)$ for $* \leq k$ as $k \to \infty$.
    Dualizing this we learn that $H_*(X)$ can be written as an inverse limit of finite $\mc{A}$-modules which stabilizes in each degree in finitely many steps.
    As finite $\mc{A}$-modules are automatically $I$-complete, this is enough to conclude that
    $H_*(X)$ is $I$-complete. 
\end{proof}

    
    



\begin{ex}
    Let $X$ be a compact, simply connected space. 
    Then by \cite[Corollary 8.7.4]{schwartz1994unstable} we have $\Omega^n X$ has finite type, locally finite $\bb{F}_p$-cohomology.
\end{ex}

\begin{ex}
    Let $X$ be a spectrum whose homology is concentrated in a bounded range. 
    The $\bb{F}_p$-homology of $X$ is $I$-complete complete.
\end{ex}

\begin{ex}
    Let $\mc{O}$ be an operad for which the space of arity $n$ operations, $\mc{O}(n)$, 
    is given by a finite $\Sigma_n$ space with free action.
    Examples of such operads include the $\bb{E}_n$-operad for $n \neq \infty$.
    The free $\mc{O}$-algebra $\rm{Free}_{\mc{O}}(X)$ on a finite, $1$-connective spectrum $X$ 
    has $I$-complete homology.
\end{ex}

\begin{proof}[Details.]
    Recall that the free $\mc{O}$-algebra on a spectrum $X$ is given by
    $$\rm{Free}_{\mc{O}}(X) \cong \bigoplus_{n} (\mc{O}(n) \otimes X^{\otimes n } )_{h\Sigma_n} .$$
    The hypotheses on $X$ and $\mc{O}$ guarantee that 
    $(\mc{O}(n) \otimes X^{\otimes n } )_{h\Sigma_n}$ 
    is a finite, $n$-connective spectrum for each $n$.
    Each of these finite pieces is automatically $I$-complete and the increasing connectivities
    ensure that their sum is $I$-complete as well.
\end{proof}


\printbibliography

\end{document}